\documentclass[10pt]{article}

\usepackage{amsmath}
\usepackage{amssymb}
\usepackage{array}
\usepackage{amsthm}
\usepackage{enumerate}
\usepackage{graphicx}
\usepackage{graphics}
\usepackage{subfigure}

\newcommand{\R}{\mathbb{R}}

\newcommand{\half}{\frac{1}{2}}

\newcommand{\Sy}{\mathbb{S}}

\newcommand{\rank}{\text{rank}}

\newcommand{\graph}{\mathsf{G}}
\newcommand{\nodeset}{{\mathsf{N}}}
\newcommand{\edgeset}{{\mathsf{E}}}

\newcommand{\edgesetF}{{\mathsf{F}}}
\newcommand{\clique}{{\mathsf{C}}}

\newcommand{\cltree}{{\cal T}}
\newcommand{\clnodeset}{{\cal N}}
\newcommand{\cledgeset}{{\cal E}}

\newcommand{\diag}{\text{diag}}
\newcommand{\range}{\text{range}}

\newcommand{\beq}{\begin{equation}}
\newcommand{\eeq}{\end{equation}}
\newcommand{\bal}{\begin{aligned}}
\newcommand{\eal}{\end{aligned}}

\newcommand{\subblk}[3]{[{#1}_{#2}]_{#3}}

\newtheorem{assumption}{Assumption}
\newtheorem{theorem}{Theorem}
\newtheorem{defn}{Definition}
\newtheorem{lemma}{Lemma}

\author{Arvind U. Raghunathan and Andrew V. Knyazev\\
Mitsubishi Electric Research Laboratories\\
201 Broadway, Cambridge, MA 02139 \\
Email:\texttt{raghunathan@merl.com},\texttt{knyazev@merl.com}}

\begin{document}

\title{Degeneracy in Maximal Clique Decomposition for Semidefinite Programs}

\maketitle

\begin{abstract}
Exploiting sparsity in Semidefinite Programs (SDP)  is critical to solving large-scale 
problems. The chordal completion based maximal clique decomposition is the preferred 
approach for exploiting sparsity in SDPs.  In this paper, we show that the 
maximal clique-based SDP decomposition is primal degenerate when the SDP has a 
low rank solution.  We also derive conditions under which the multipliers in the maximal 
clique-based SDP formulation is not unique.  Numerical experiments demonstrate that the 
SDP decomposition results in the schur-complement matrix  of the Interior Point 
Method (IPM) having higher condition number than for the original SDP formulation.
\end{abstract}


\section{Introduction}\label{sec:intro}

Semidefinite programming (SDP) is a subfield of convex optimization concerned with the 
optimization of a linear objective function over the intersection of the cone of 
symmetric positive semidefinite matrices with an affine space. 
Many problems in operations research and combinatorial 
optimization can be modeled or approximated as SDPs~\cite{deklerkbook,sdpreview}. 
For an SDP defined over the set of $n \times n$ symmetric matrices 
the number of unknowns in the problem grows as $O(n^2)$.  
Since the seminal work of Nesterov and Nemirovskii~\cite{nesterovbook87}, 
Interior Point Methods (IPMs) have become the preferred approach for solving SDPs.  
The complexity of the step computation in IPM is typically 
$O(mn^3 + m^2n^2)$~\cite{TodTohTut98}.  

Given the quadratic, cubic growth in $m,n$ 
of the computational cost respectively, it is imperative to exploit problem structure in solving large-scale 
SDPs.  For SDPs modeling practical applications, the data matrices involved are 
typically \emph{sparse}.  Denote by,  
$\nodeset = \{1,\ldots,n\}$ and by $\edgeset = \{ (i,j) \;|\; i \neq j,\; (i,j)-\text{th entry of some data 
matrix is non-zero}\}$. The set $\edgeset$, also called the \emph{aggregate sparsity 
pattern}, represents the non-zero entries in the objective and constraint matrices, 
that is the sparsity in the problem data. Consequently, only the entries of the matrix variable 
corresponding to the aggregate sparsity pattern are involved in the problem.  From the 
computational stand-point it is desirable to work only with such entries to reduce the 
number of unknowns in the problem from $O(n^2)$ to $O(|E|)$.  However, the 
semidefinite constraint couples all of the entries of the symmetric matrix.  
Fukuda et al~\cite{FukKojMur00} exploit the result of Grone et al~\cite{GroJohSa84} to 
decompose the SDP defined on $n \times n$ symmetric matrices into smaller sized matrices.  
Grone et al~\cite[Theorem 7]{GroJohSa84}  states that for a graph 
$\graph(\nodeset,\edgeset)$ that is chordal:  
the positive semidefinite condition on $n \times n$ matrix is equivalent to  
positive semidefinite condition on submatrices corresponding to the maximal cliques that 
cover all the nodes and edges in the graph $\graph(\nodeset,\edgeset)$.  
Nakata et al~\cite{NakFujFuk03} 
implemented the decomposition within a SDP software package SDPA~\cite{sdpa6} and 
demonstrated the scalability of the approach.  More recently, the authors of SDPA 
have also extended the implementation to take advantage of multi-core 
architectures~\cite{sdpa7,YamNak13}.  More recently, Kim and Kojima~\cite{KimKoj12} extended this 
approach for solving semidefinite relaxations of polynomial optimization problems.  

\subsection{Our Contribution}
In this paper, we study the properties of the \emph{conversion approach} 
of~\cite{FukKojMur00,NakFujFuk03} which converts the original SDP into an SDP with 
multiple sub-matrices and additional equality constraints.  We show that the 
SDP resulting from the {conversion approach} is \emph{primal degenerate} 
when the SDP solution has low-rank.  We show that this 
can occur even when the solution to the original SDP is \emph{primal non-degenerate}. 
Thus, this degeneracy is a consequence of the conversion approach.  We also derive 
conditions under which the dual multipliers are \emph{non-unique}. 
We demonstrate through numerical experiments that condition numbers of 
schur-complement matrix of IPM are much higher for the conversion approach as compared 
with the original SDP formulation.  To the best of our knowledge, this is the first result 
describing the degeneracy of the conversion approach.

The rest of the paper is organized as follows. \S\ref{sec:sdpclique} introduces the SDP 
formulation and the maximal clique decomposition.  The conversion approach~\cite{FukKojMur00} 
is described in \S\ref{sec:conversionsdp}.  
\S\ref{sec:primaldegen} proves the primal degeneracy and dual non-uniqueness of the 
conversion approach.  Numerical experiments validating the results are presented 
in \S\ref{sec:numericalexpts}, followed by conclusions in \S\ref{sec:conclusions}. 

\subsection{Notation}
In the following, $\R$ denotes the set of reals and $\R^n$ is the space of $n$ dimensional 
column vectors.  For a vector $x \in \R^n$, $\subblk{x}{}{i}$ denotes the $i$-th component of 
$x$ and $0_{n} \in \R^n$ denotes the zero vector, $e_i \in \R^n$ the vector with 
$1$ for the $i$-th component and $0$ otherwise.  The notation 
$\diag(\lambda_1,\ldots,\lambda_n)$ denotes a diagonal matrix with the values $\lambda_i$ 
on the  diagonal.  Given a vector $v \in \R^n$ and subset $\clique \subseteq \{1,\ldots,n\}$, 
$v_{\clique}$ denotes the subvector composed of $\subblk{v}{}{i}$ for $i \in \clique$. 
$\Sy^n$ denotes the set of $n \times n$ real symmetric matrices and 
$\Sy^n_+$ ($\Sy^n_{++}$) denotes the set of $n\times n$ real symmetric positive semi-definite 
(definite) matrices.  
Further, $A \succeq (\succ) 0$ denotes that $A \in \Sy^n_+ (\Sy^n_{++})$.  
For a matrix $A \in \Sy^n$,  $\subblk{A}{}{ij}$ denotes the $(i,j)$-th entry of the matrix $A$ and 
$\rank(A)$ denotes the rank of $A$.  For a pair of matrices $A_1,A_2 \in \Sy^n$, $\range\{A_1,A_2\}$ 
denotes the subspace of  symmetric matrices spanned by $A_1,A_2$. Denote by $\nodeset = \{1,\ldots,n\}$.  
The notation $A \bullet B$ denotes the standard trace inner product between symmetric matrices 
$A \bullet B = \sum_{i=1}^n\sum_{j=1}^n\subblk{A}{}{ij}\subblk{B}{}{ij}$ for $A,B \in \Sy^n$.  
For sets $\clique_s, \clique_t \subseteq \nodeset$ and $A \in \Sy^n$, 
$A_{\clique_s\clique_t}$ is a $|\clique_s| \times |\clique_t|$ submatrix of $A$ formed by 
removing rows and columns of $A$ that are not in $\clique_s$, $\clique_t$ respectively. 

\subsection{Background on Graph Theory~\cite{BlaPey93}}
In this paper we only consider undirected graphs.  
Given a graph $\graph(\nodeset,\edgesetF)$, a \emph{cycle in } $\edgesetF$ is a sequence of vertices 
$\{i_1,i_2,\ldots,i_q\}$ such that $i_j \neq i_{j'}$, $(i_j,i_{j+1}) \in \edgesetF$ and $(i_q,i_1) 
\in \edgesetF$.  The cycle in $\edgesetF$ with $q$ edges is called a cycle of length $q$.
Given a cycle $\{i_1,\ldots,i_q\}$ in $\edgesetF$, 
a \emph{chord} is an edge $(i_j, i_{j'})$ for $j,j'$ that are non-adjacent in the cycle.
A graph $\graph(\nodeset,\edgesetF)$ is said to be \emph{chordal} if every cycle of length 
greater than $3$ has a chord.
Given $\graph(\nodeset,\edgesetF)$, $\edgesetF' \supseteq \edgesetF$ is called 
a \emph{chordal extension} if the graph $\graph'(\nodeset,\edgesetF')$ is chordal.
Given a graph $\graph(\nodeset,\edgesetF)$, $\clique \subset \nodeset$ is called a 
\emph{clique} if it satisfies the property that $(i,j) \in \edgesetF$ for all $i,j \in \clique$. 
A clique $\clique$ is \emph{maximal} if there does not exist clique 
$\clique' \supset \clique$.  
For a chordal graph, the maximal cliques can be arranged as a tree, called  \emph{clique tree},  
$\cltree(\clnodeset,\cledgeset)$ in which $\clnodeset = \{\clique_1,\ldots,\clique_\ell\}$ and 
$(\clique_s,\clique_t) \in \cledgeset$ are edges between the cliques.  The interested reader is 
referred to~\cite{BlaPey93} for an in-depth exposition of chordal graphs and clique trees.

\subsection{Matrix Terminology~\cite{FukKojMur00}}
For a set $\edgesetF \subset \nodeset 
\times \nodeset$, $\Sy^n(\edgesetF)$ denotes the set of symmetric 
$n \times n$ matrices with only entries in $\edgesetF$ specified. 
A matrix $\bar{X} \in \Sy^n(\edgesetF)$ 
is called a \emph{symmetric partially specified matrix}. 
A matrix $X \in \Sy^n$ is called a 
\emph{completion} of  $\bar{X} \in \Sy^n(\edgesetF)$  if $\subblk{X}{}{ij} = \subblk{\bar{X}}{}{ij}$ 
for all $(i,j) \in \edgesetF$.  
A completion $X \in \Sy^n$ of $\bar{X} \in \Sy^n(\edgesetF)$ that is positive 
semidefinite (definite) is said to be a \emph{positive semidefinite (definite) completion} 
of $\bar{X}$.  

\section{Maximal Clique Decomposition in SDP}\label{sec:sdpclique}

Consider the following SDP:
\begin{equation}
\begin{aligned}
\min\limits_{X \in \Sy^n} \;& A_0\bullet X \\
\mbox{s.t.} \;& A_p\bullet X = b_p \;\forall\; p = 1,\ldots,m \\
\; & X \succeq 0
\end{aligned}\label{sdp}
\end{equation}
where $A_p \in \Sy^n$. Denote by $\edgeset = \{ (i,j) \;|\; i \neq j,\; \subblk{A}{p}{ij} \neq 0 \text{ for 
some } 0 \leq p \leq m\}$.  
The set $\edgeset$, also called the \emph{aggregate sparsity pattern}~\cite{FukKojMur00}, 
represents the non-zero entries in the objective and constraint matrices, that is 
the sparsity in the problem data.  Clearly, only the entries 
$\subblk{X}{}{jk}$ for $(j,k) \in \edgeset$ feature in the objective and equality 
constraints in~\eqref{sdp}.  In a number of practical applications, 
$|\edgeset| << n^2$.  From a computational standpoint, it is 
desirable to work only with $\subblk{X}{}{jk}$ for $(j,k) \in \edgeset$. In other words, 
we want to solve
\begin{equation}
\begin{aligned}
\min\limits_{\bar{X} \in \Sy^n(\edgeset)} \;& \sum\limits_{(i,j) \in \edgeset} 
\subblk{A_0}{}{ij} \subblk{\bar{X}}{}{ij} \\
\mbox{s.t.} \;& \sum\limits_{(i,j) \in \edgeset} \subblk{A_p}{}{ij} \subblk{\bar{X}}{}{ij} = 
b_p \;\forall\; p = 1,\ldots,m \\
\; & \bar{X} \text{ has a positive semidefinite completion.}
\end{aligned}\label{sdpsparse}
\end{equation}
The result of~\cite{GroJohSa84} provides the conditions under which such a 
completion exist.  We state this below in a form convenient for further 
development as in~\cite[Theorem\;2.5]{FukKojMur00}.

\begin{lemma}[{\cite[Theorem\;2.5]{FukKojMur00}}]\label{thm:cliquethm} 
Let $\graph(\nodeset,\edgesetF)$ be a chordal graph and let 
$\{\clique_1,\ldots.\clique_\ell\}$ be the family of all maximal cliques.  
Then, $\bar{X} \in \Sy^n(\edgesetF)$ has a positive semidefinite (definite) completion 
if and only if $\bar{X}$ satisfies 
\beq
\bar{X}_{\clique_s\clique_s} \succeq 0 (\succ 0) \;\forall\; s = 1,\ldots,\ell. \label{cliquepsdpd}
\eeq
\end{lemma}  

Using Lemma~\ref{thm:cliquethm} Fukuda et al~\cite{FukKojMur00} proposed the 
\emph{conversion approach} which we describe next. 

\section{Conversion Approach}\label{sec:conversionsdp}
Given the graph $\graph(\nodeset,\edgeset)$, with $\edgeset$ the aggregate sparsity 
pattern of SDP~\eqref{sdp}, the conversion approach~\cite{FukKojMur00} proceeds by : (a) computing a chordal extension $\edgesetF \supseteq \edgeset$; (b) the set of maximal cliques 
$\clnodeset = \{\clique_1,\ldots,\clique_\ell\}$ of the graph $\graph(\nodeset,\edgesetF)$ 
are identified; (c) the clique tree $\cltree(\clnodeset,\cledgeset)$ is computed; 
and (d) a SDP is posed in terms of matrices defined on the set of maximal 
cliques that is equivalent to SDP in~\eqref{sdp}. 
Additional equality constraints are introduced to equate the overlapping entries in the 
maximal cliques.  
Prior to stating the SDP formulation we introduce notation that facilitates further 
development.  Denote,
\beq\bal
\sigma_s :&\; \nodeset \rightarrow \{1,\ldots,|\clique_s|\}  \text{ mapping the original} \\
&\; \text{indices to the ordering in the clique } \clique_s \\
\subblk{A_{s,p}}{}{\sigma_s(i)\sigma_s(j)} =&\; \left\{ \bal 
\subblk{A_p}{}{ij} &\; \text{if } s = \min\{t \; |\; (i,j) \in \clique_t\} \\
0 		             &\; \text{otherwise}
\eal\right. \\
E_{s,ij} =&\; \half \left(e_{\sigma_s(i)} e_{\sigma_s(j)}^T + e_{\sigma_s(j)} e_{\sigma_s(i)}^T 
\right) \forall i,j \in \clique_s \\
(s,t) \in&\; \cltree \iff (\clique_s,\clique_t) \in  \cledgeset \\
\clique_{st} =&\; \clique_s \cap \clique_t 
\eal\eeq
where $e_{\sigma_s(i)} \in \R^{|\clique_s|}$.  
The SDP in~\eqref{sdp} can be equivalently posed using the above notation as,
\begin{equation}
\begin{aligned}
\min\limits_{X_s \in \Sy^{|\clique_s|}} \;& \sum\limits_{s=1}^{\ell}A_{s,0}\bullet X_s && \\
\mbox{s.t.} \;& \sum\limits_{s=1}^\ell A_{s,p}\bullet X^s = b_p &&\forall\; p = 1,\ldots,m \\
\;&  E_{s,ij}\bullet X_s = E_{t,ij}\bullet X_t && \forall\; i \leq j, i,j \in \clique_{st}, \\
\;& && (s,t) \in \cledgeset \\
\;& X_s \succeq 0 && \forall\; s = 1,\ldots,\ell.
\end{aligned}\label{sdpconv}
\end{equation}
We refer to the SDP in~\eqref{sdpconv} as the \emph{conversion SDP}.  
In the following we present and analyze an analogue of the conversion approach for linear 
programming.  The analysis shows that primal degeneracy results for the linear program under 
certain assumptions.  Formal proofs for primal degeneracy of conversion SDP are left for \S\ref{sec:primaldegen}.  

\subsection{Intuition for Primal Degeneracy}

To motivate the primal degeneracy of conversion SDP we 
provide a conversion approach inspired decomposition for linear programs (LPs).  
Consider a LP of the form,
\begin{equation}
\begin{aligned}
\min\limits_{x \in \R^n}\;&a_0^Tx \\
\text{s.t.}\;&a_p^Tx = b_p \;\forall\; p = 1,\ldots,m \\
\;& x \geq 0
\end{aligned}\label{lp}
\end{equation}
where  $a_p \in \R^n$ and $b \in \R^m$.  Suppose we decompose the LP \eqref{lp} 
using the sets in $\{\clique_1,\ldots,\clique_\ell\}$ as,
\begin{equation}
\begin{aligned}
\min\limits_{x_s \in \R^{|\clique_s|}}\;&\sum\limits_{s=1}^la_{0,s}^Tx_s \\
\text{s.t.}\;&\sum\limits_{s=1}^la_{s,p}^Tx_s = b_p \;\forall\; p = 1,\ldots,m \\
\;& 
\subblk{x_{s}}{}{\sigma^{\mathrm{LP}}_s(i)} = 
\subblk{x_t}{}{\sigma^{\mathrm{LP}}_t(i)}  
\;\forall\; i \in \clique_{st}, (s,t) \in \cledgeset  \\
\;& x_s \geq 0 \;\forall\; s = 1,\ldots,\ell \\
\end{aligned}\label{lpconv}
\end{equation}
where 
\[\bal
& \sigma^{\mathrm{LP}}_s : \nodeset \rightarrow \{1,\ldots,|\clique_s|\} \\
& \subblk{a_{s,p}}{}{\sigma^{\mathrm{LP}}_{s}(i)} = \left\{ \begin{array}{cl} 
\subblk{a_p}{}{i} & \text{ if } s = \min \{ t | i \in \clique_t \} \\
0 & \text{ otherwise}  
\end{array} \right. .
\eal\]
With the above definition of the matrices it is easy to see that the LPs in~\eqref{lp} and~
\eqref{lpconv} are equivalent.  Further, if $x^*$ is an optimal solution to LP~\eqref{lp} then, 
$x^*_s = x^*_{\clique_s}$ is optimal for~\eqref{lpconv}.  
Suppose, there exists $i \in \clique_s \cap \clique_t$ for which $\subblk{x^*}{}{i} = 0$ then, the 
set of constraints 
\[\bal
\subblk{x_s}{}{\sigma^{\mathrm{LP}}_s(i)} = 
\subblk{x_t}{}{\sigma^{\mathrm{LP}}_t(i)} \\
\subblk{x_s}{}{\sigma^{\mathrm{LP}}_s(i)} = 0,  
\subblk{x_t}{}{\sigma^{\mathrm{LP}}_t(i)} = 0
\eal\]
are linearly dependent.  The linear dependency of the constraints can be avoided if the nonnegative bounds on shared entries are enforced exactly once for each index $i$.  
For example, the non negativity constraints in~\eqref{lpconv} can be enforced for each 
$i \in \nodeset$:
\[
\subblk{x_{s}}{}{\sigma^{\mathrm{LP}}_s(i)} \geq 0 \text{ if } s = \min \{ t | i \in \clique_t \}.
\]
In summary the degeneracy occurs due to a shared element activating the bound at the solution.  
In a direct analogy, the conversion SDP in \eqref{sdpconv} is primal degenerate when, 
$\rank(X_{\clique_{st}\clique_{st}}) < |\clique_{st}|$.  This degeneracy is directly attributable to 
the duplication of the semidefinite constraints for the submatrix 
$X_{\clique_{st}\clique_{st}}$ in both 
$X_s \succeq 0$ and $X_t \succeq 0$ for every pair of $(s,t) \in \cltree: s \neq t$.  
Unfortunately, the duplication of the semidefinite constraints cannot be avoided in the case of 
SDP without losing the linearity.  We provide formal arguments for the degeneracy and dual 
multiplicity of the conversion SDP~\eqref{sdpconv} in the following section. 

\section{Primal Degeneracy \& Dual Non-uniqueness of Conversion Approach}
\label{sec:primaldegen}

We review the conditions for primal non-degeneracy and dual uniqueness
 for the SDP~\eqref{sdp} introduced by Alizadeh et al~\cite{AliHaeOve97} in 
 \S\ref{sec:primalnondegensdp}.  We also extend this notion to that of the conversion SDP~\eqref{sdpconv}.  
 \S\ref{sec:sdpconvdegen} proves the primal  degeneracy result for the conversion approach 
 and \S\ref{sec:sdpconvmultidual} proves the dual non-uniqueness.  

\subsection{Primal Nondegeneracy and Dual Uniqueness in SDPs}
\label{sec:primalnondegensdp}

Suppose $X \in \Sy^n$ with $\rank(X) = r$ with eigenvalue decomposition 
$X = Q\diag(\lambda_1,\ldots,\lambda_r,0,\ldots,0)Q^T$ then, the tangent space  
${\mathbb T}_X$ of rank-$r$ symmetric matrices is 
\beq
{\mathbb T}_X = \left\{ \left. Q \begin{bmatrix} U & V \\ V^T & 0 \end{bmatrix}Q^T  \right| 
U \in \Sy^{r}, V \in \R^{r \times (n-r)}  
\right\} \label{tangentspace}
\eeq
and the space orthogonal to ${\mathbb T}_X$ is given by,
\beq
{\mathbb T}^{\perp}_X = \left\{ \left. Q \begin{bmatrix} 0 & 0 \\ 0 & W \end{bmatrix}Q^T  \right| 
W \in \Sy^{n-r} \right\}. \label{orthtangentspace}
\eeq
The null space of equality constraints ${\mathbb N}_A$ is,
\beq
{\mathbb N}_A = \left\{ Y \in \Sy^n | A_p \bullet Y = 0 \;\forall\; p = 1,\ldots,m\right\}
\label{nullspace}
\eeq
and the space orthogonal to ${\mathbb N}_A$ is given by,
\beq
{\mathbb N}^{\perp}_A = \text{span}\{\{A_p\}_{p=1}^{m}\}. 
\label{orthnullspace}
\eeq

\begin{defn}[\cite{AliHaeOve97}]\label{def:nondegen} Suppose $X$ is primal feasible for \eqref{sdp} with 
$\rank(X) = r$, then $X$ is primal nondegenerate if the following equivalent conditions hold:
\begin{subequations}\label{nondegen}
\begin{align}
 {\mathbb T}_X + {\mathbb N}_{A} &= \Sy^n 	\label{nondegen-fp} \\
 {\mathbb T}^{\perp}_X \cap {\mathbb N}^{\perp}_{A} &= \{\emptyset\} \label{nondegen-orth} \\
\sum\limits_{p=1}^m \alpha_p A_p &\in {\mathbb T}^{\perp}_X \implies \alpha_p = 0 \,\forall\, p = 1,\ldots,m. \label{nondegen-licq}
\end{align}
\end{subequations}
\end{defn}

\begin{lemma}[{\cite[Theorem\;2]{AliHaeOve97}}]
Suppose $X^* \in \Sy^n$ is primal nondegenerate and optimal for \eqref{sdp}.  Then the optimal 
dual multipliers $(\zeta^*,S^*)$, for the equality and positive semidefinite constraints respectively, 
satisfying the first order optimality conditions~\eqref{firstoptsdp} are unique,
\beq\bal
A_0 + \sum\limits_{p=1}^m\zeta^*_pA_p - S^* = &\; 0 \\
A_p \bullet X^* =&\; b_p \;\forall\; p = 1,\ldots,m \\
X^*,S^* \succeq 0, X^*S^* =&\; 0.
\eal\label{firstoptsdp}\eeq
\end{lemma}

We extend the notions of non-degeneracy to the conversion SDP~\eqref{sdpconv} in the following. 
The tangent space for the matrices $X_s$ is defined as,
\beq
{\mathbb T}_{s,X} = \left\{ \left. Q_s \begin{bmatrix} U & V \\ V^T & 0 \end{bmatrix}Q_s^T  \right| 
U \in \Sy^{r_s}, V \in \R^{r_s \times (|\clique_s|-r_s)}  
\right\} \label{tangentspaceXs}
\eeq
and the space orthogonal to ${\mathbb T}_{s,X}$ is
\beq
{\mathbb T}_{s,X}^{\perp} = \left\{ \left. Q_s \begin{bmatrix} 0 & 0 \\ 0 & W \end{bmatrix}Q_s^T  \right| 
W \in \Sy^{|\clique_s|-r_s}  
\right\} \label{orthtangentspaceXs}
\eeq
where $X_s =$ $Q_s\diag(\lambda_1,\ldots,\lambda_{r_s},
0,\ldots,0)$ $(Q_s)^T$ and $r_s = \rank(X_s)$. The tangent space for the  
conversion SDP~\eqref{sdpconv} is denoted by ${\mathbb T}$ and the space orthogonal to the tangent space  is 
denoted by ${\mathbb T}^{\perp}$ are,
\[ {\mathbb T} = \times_{s=1}^\ell {\mathbb T}_{s,X} \text{ and }
{\mathbb T}^{\perp} = \times_{s=1}^{\ell} {\mathbb T}^{\perp}_{s,X}.
\]
The null space of equality constraints for the conversion SDP~\eqref{sdpconv} is,
\beq
{\mathbb N} = \left\{ \times_{s=1}^{\ell}Y_s \in \Sy^{|\clique_s|} \left| \begin{array}{l}
\sum\limits_{s=1}^{\ell}A_{s,p} \bullet Y_s = 0 \;\forall\; p = 1,\ldots,m \\
E_{s,ij}\bullet Y_s = E_{t,ij}\bullet Y_t \\
\forall\; i\leq j, i,j \in \clique_{st}, (s,t) \in \cledgeset 
\end{array} \right.
\right\}
\label{nullspaceXs}
\eeq
The nullspace for conversion SDP~\eqref{sdpconv} couples all the matrices corresponding to cliques.  The space 
orthogonal to ${\mathbb N}$ denoted by ${\mathbb N}^{\perp}$ is a 
lot more convenient to work with since it can be written as product of spaces.  
The range space of the constraint matrices corresponding to a particular clique $\clique_s$ is given by,
\beq
{\mathbb N}^{\perp}_s = \left\{ \left. \begin{aligned}
&\sum\limits_{p=1}^m\alpha_{p} A_{s,p} \\
&+ \sum\limits_{t:(s,t) \in \cledgeset}\sum\limits_{i \leq j \in \clique_{st}} \beta_{st,ij} E_{s,ij}  \\
&- \sum\limits_{t:(t,s) \in \cledgeset}\sum\limits_{i \leq j \in \clique_{ts}} \beta_{ts,ij} E_{s,ij} \end{aligned} \right|
\begin{aligned}
&\text{for some } \alpha_p \\
&p=1,\ldots,m, \\
&\beta_{st,ij} \, (s,t) \in\cledgeset, \\
&i \leq j \in \clique_{st}  \\
&\text{not all } 0
\end{aligned}
\right\}
\eeq
The range space of the conversion SDP is,
\beq
{\mathbb N}^{\perp} = \left\{ \times_{s=1}^{\ell} {\mathbb N}^{\perp}_s \left| 
\begin{aligned}
&\text{for some } ( \{\alpha_p\}_{p=1}^m,  \\
&\{\beta_{st,ij}\}_{ (s,t) \in\cledgeset, i \leq j \in \clique_{st}} ) \neq 0
\end{aligned} \right.
\right\}.
\eeq
Analogous to Definiton~\ref{def:nondegen}, the conditions for primal non-degeneracy of conversion SDP 
is stated below. 

\begin{defn}\label{def:nondegenconvsdp} Suppose $\{X_s\}_{s=1}^{\ell}$ is primal feasible for \eqref{sdpconv} with $\rank(X_s) = r_s$, then $\{X_s\}$ is primal nondegenerate if the following equivalent conditions hold:
\begin{subequations}\label{convsdpnondegen}
\begin{align}
 {\mathbb T} + {\mathbb N} &= \Sy^{\bar{n}} 	\label{convsdpnondegen-fp} \\
 {\mathbb T}^{\perp} \cap {\mathbb N}^{\perp} &= \{\emptyset\} \label{convsdpnondegen-orth} 
\end{align}
\end{subequations}
where $\bar{n} = \sum_{s=1}^{\ell} \half |\clique_s|(|\clique_s|+1)$.  
\end{defn}

\subsection{Primal Degeneracy of Conversion SDP}
\label{sec:sdpconvdegen}

\begin{assumption}\label{ass:smallrankofX}
The SDP~\eqref{sdp} has an optimal solution $X^*$ with 
$\rank(X^*) < |\clique_{st}|$ for some $(s,t) \in \cledgeset$.
\end{assumption}

In the following we denote by $X^{*}_s$ the optimal solution to the conversion 
SDP~\eqref{sdpconv}.  The following result is immediate. 

\begin{lemma}\label{lem:bndrankofXs} 
$\rank(X^{*}_s) \leq \rank(X^*) \;\forall\; s = 1,\ldots,\ell$.
\end{lemma}
\begin{proof}
By definition, $X^{*}_s = X^*_{\clique_s\clique_s}$ is a principal submatrix of $X^*$.  The claim 
follows by noting that the rank of any principal sub-matrix cannot exceed that of the 
original matrix.
\end{proof}

\noindent The following result characterizes the eigenvectors of 
the matrices $X^{*}_s, X^{*}_t$ for cliques $s,t$  satisfying Assumption~\ref{ass:smallrankofX}. 
Without loss of generality and for ease of presentation, we assume that the shared nodes of 
$\clique_s, \clique_t$ are ordered as,
\beq
\sigma_s(i) = \sigma_t(i) \text{ and } 
1 \leq \sigma_{s}(i) \leq |\clique_{st}| \;\forall\; i \in \clique_{st}. \label{stpermute}
\eeq

\begin{lemma}\label{lem:eigenvectorXst}
Suppose Assumption~\ref{ass:smallrankofX} holds for cliques $s,t$.  Then, there exists 
$u \in \R^{|\clique_{st}|}$ such that $v_s = [u^T \; 0^T_{|\clique_s\setminus\clique_{st}|}]^T$ is 
a $0$-eigenvector of $X^{*}_s$ and $v_t = [u^T \; 0^T_{|\clique_t\setminus\clique_{st}|}]^T$ is 
a $0$-eigenvector of $X^{*}_t$.
\end{lemma}
\begin{proof}
From Lemma~\ref{lem:bndrankofXs}, we have that $\rank(X^{*}_s), \rank(X^{*}_t) \leq \rank(X^*) < 
|\clique_{st}|$ where the second inequality follows from Assumption~\ref{ass:smallrankofX}. 
Applying Lemma~\ref{lem:bndrankofXs} $X_s^*, X_t^*$, it is also true that the submatrix of 
$X^{*}_s, X^{*}_t$ corresponding to $\clique_{st}$ must have rank smaller than $|\clique_{st}|$.  
Hence, there exists a vector $u \in \R^{|\clique_{st}|}$ that lies in the nullspace of the principal submatrix 
$X_{\clique_{st}\clique_{st}}$.   Defining the vector 
$v_s = [ u^T \; 0^T_{|\clique_s\setminus\clique_{st}|}]^T$ and  
taking the right and left products with $v_s$ of the matrix $X^{*}_s$,
\[\bal
v_s^T X_s^{*}v_s = u^T(X^*_{s})_{\clique_{st}\clique_{st}}u = 0 \implies \frac{v_s^TX^{*}_sv_s}{v_s^Tv_s} = 0  \\
\implies v_s \text{ is in the span of  } 0-\text{eigenvectors of } X^{*}_s \\
\implies v_s \text{ is a } 0-\text{eigenvector of } X^{*}_s
\eal\]
The claim on $X^{*}_t$ can be proved similarly and this completes the proof.
\end{proof}

We can now state the main result on the primal degeneracy of the conversion 
SDP~\eqref{sdpconv}.

\begin{theorem}\label{thm:sdpconvisdegenerate}
Suppose Assumption~\ref{ass:smallrankofX} holds. Then, the solution $X^{*}_s$ 
of the conversion SDP~\eqref{sdpconv} is primal degenerate.
\end{theorem}
\begin{proof}
Suppose there exists scalars $\alpha_p, \beta_{st,ij} \neq 0$ such that 
\beq
\begin{aligned}
\sum\limits_{p=1}^m\alpha_{p} A_{s,p} + \sum\limits_{(s,t) \in \cledgeset}\sum\limits_{i \leq j \in 
\clique_{st}} \beta_{st,ij} E_{s,ij} \in {\mathbb T}_{s,X^*}^{\perp} \\
\sum\limits_{p=1}^m\alpha_p A_{t,p} - \sum\limits_{(s,t) \in \cledgeset}\sum\limits_{i \leq j \in 
\clique_{st}} \beta_{st,ij} E_{t,ij} \in {\mathbb T}_{t,X^*}^{\perp}
\end{aligned}\label{convsdpdegen-licq}
\eeq
holds.  Then, we have that $X^{*}_s$ is primal degenerate since the condition in~\eqref{convsdpnondegen-orth} 
does not hold.  In the following we will denote by $\hat{s}, \hat{t}$ \emph{a pair} of cliques satisfying 
Assumption~\ref{ass:smallrankofX}.  We will show in the following that:
 \beq
\alpha_p = 0 , \beta_{st,ij} = \left\{\begin{aligned}
\hat{\beta}_{\hat{s}\hat{t},ij} &\; \text{for } (s,t) = (\hat{s},\hat{t}) \\
0 &\; \text{otherwise} 
\end{aligned} \right. \label{degenchoice}
\eeq
satisfies~\eqref{convsdpdegen-licq}.  The choice in~\eqref{degenchoice} results in the left hand side of~\eqref{convsdpdegen-licq} evaluating to $0$ for all $(s,t) \neq (\hat{s},\hat{t})$.  Thus,~\eqref{convsdpdegen-licq} 
holds trivially for all $(s,t) \neq (\hat{s},\hat{t})$ since $0 \in {\mathbb T}^{\perp}_{s,X^*}, {\mathbb T}^{\perp}_{t,X^*}$.   
In the rest of the proof we will consider only the cliques $(s,t) = (\hat{s}, \hat{t})$.  

By Lemma~\ref{lem:eigenvectorXst} and~\eqref{orthtangentspaceXs},
$v_sv_s^T \in {\mathbb T}_{s,X^*}^{\perp} \text{ and } v_tv_t^T \in {\mathbb T}_{t,X^*}^{\perp}$.
Define $\hat{\beta}_{st,ij} = E_{s,ij}\bullet (v_sv_s^T) = v_s^TE_{s,ij}v_s$.  By definition
of $\hat{\beta}_{st,ij}$, 
\beq
v_sv_s^T = \sum_{i \leq j \in \clique_{st}} \hat{\beta}_{st,ij} E_{s,ij}. \label{defvsvsT}
\eeq  
By definition of $v_s, v_t$ in Lemma~\ref{lem:eigenvectorXst} and~\eqref{stpermute}, we also have that 
$v_tv_t^T =  \sum_{i \leq j \in \clique_{st}} \hat{\beta}_{st,ij} E_{t,ij}$.  Thus, the choice in~\eqref{degenchoice} 
satisfies~\eqref{convsdpdegen-licq}.  This completes the proof.
\end{proof}

\subsection{Dual Non-uniqueness in Conversion SDP}
\label{sec:sdpconvmultidual}

The solution $X^{*}_s$ and multipliers $\zeta^{*}_{s,p}, \xi^{*}_{st,ij}, S^{*}_s$ 
satisfy the first order optimality conditions for the conversion SDP~\eqref{sdpconv} for 
all $s = 1,\ldots,\ell$,
\beq\bal
& A_{s,0} + \sum\limits_{p=1}^m \zeta^{*}_{s,p}A_{s,p} + \sum\limits_{t : (s,t) \in \cledgeset} 
\sum\limits_{i \leq j \in \clique_{st}} \xi^{*}_{st,ij} E_{s,ij} \\
& - \sum\limits_{t: (t,s) \in \cledgeset} 
\sum\limits_{i \leq j \in \clique_{ts}} \xi^*_{ts,ij} E_{s,ij} - S^{*}_s = 0 \\
& \sum\limits_{s=1}^\ell A_{s,p}\bullet X^{*}_s = b_p  \\
& E_{s,ij}\bullet X^{*}_s = E_{t,ij}\bullet X^{*}_t  \\
& X^{*}_s,S^{*}_s \succeq 0, X^{*}_sS^{*}_s = 0. 
\eal\label{firstoptconvsdp}\eeq

\begin{theorem}\label{thm:nonuniquemultiplier}
Suppose Assumption~\ref{ass:smallrankofX} holds and $v_s^TS^*_sv_s > 0$ or  
$v_t^TS^*_tv_t > 0$.  Then, the optimal multipliers 
for the conversion SDP~\eqref{sdpconv} are not unique.
\end{theorem}
\begin{proof}
Let $\zeta^*_{s,p}, \xi^*_{st,ij}, S^*_s$ satisfy the first order optimality 
conditions~\eqref{firstoptconvsdp} for the conversion SDP~\eqref{sdpconv}. 
In the following we show by construction the existence of other multipliers satisfying the 
conditions in~\eqref{firstoptconvsdp} for the cliques $(\hat{s},\hat{t})$ as defined in 
Theorem~\ref{thm:sdpconvisdegenerate}.  In the rest of the proof $(s,t) = (\hat{s},\hat{t})$. 
Suppose, $v_s^TS^*_sv_s = \gamma > 0$.  Since, $v_s$ is a 0-eigenvector of $X^*_s$ 
(Lemma~\ref{lem:eigenvectorXst}) and 
$X^*_sS^*_s = 0$~\eqref{firstoptconvsdp} we have that $v_s$ is also an eigenvector 
of $S^*_s$.  Thus, for all $0 \leq \delta \leq \gamma$, 
\beq\bal
X^*_s(S_s^* - \delta v_sv_s^T) = 0,  S_s^* - \delta v_sv_s^T \succeq 0 \\
X^*_t(S_t^* + \delta v_tv_t^T) = 0,  S_t^* + \delta v_sv_s^T \succeq 0  
\eal\label{complXs}\eeq
Following the proof of Theorem~\ref{thm:sdpconvisdegenerate} we have that there exist 
$\hat{\beta}_{st,ij}$ such that~\eqref{defvsvsT} holds.  Hence, 
\[\bal
& \sum\limits_{i \leq j \in \clique_{st}} (\xi^*_{st,ij} - \delta\hat{\beta}_{st,ij}) E_{s,ij} - (S_s^* 
- \delta v_sv_s^T) \\
= & \sum\limits_{i \leq j \in \clique_{st}} \xi^*_{st,ij}E_{s,ij} - S_s^*. 
\eal\]
Further, by Lemma~\ref{lem:eigenvectorXst} we also have that,
\[\bal
& - \sum\limits_{i \leq j \in \clique_{st}} (\xi^*_{st,ij} - \delta\hat{\beta}_{st,ij}) E_{t,ij} - (S_t^* 
+ \delta v_tv_t^T) \\
= & - \sum\limits_{i \leq j \in \clique_{st}} \xi^*_{st,ij} E_{t,ij} - S_t^* \\
& + \left(\sum\limits_{i \leq j \in \clique_{st}} \delta\hat{\beta}_{st,ij} E_{t,ij} -  \delta v_tv_t^T 
\right)  \\
= & - \sum\limits_{i \leq j \in \clique_{st}} \xi^*_{st,ij} E_{t,ij} - S_t^* + \delta ( v_tv_t^T - v_tv_t^T) \\
= & - \sum\limits_{i \leq j \in \clique_{st}} \xi^*_{st,ij} E_{t,ij} - S_t^*.
\eal\]
Thus, for any $0 < \delta \leq \gamma$ replacing $\xi^*_{st,ij}, S^*_s,S^*_t$ with 
\[
\xi^*_{st,ij} + \delta\hat{\beta}_{st,ij}, S^*_s - \delta v_sv_s^T, S_t^*+\delta v_tv_t^T
\]
will also result in satisfaction of the first order optimality conditions 
in~\eqref{firstoptconvsdp}.  Hence, the 
multipliers are not unique when $v_s^TS_s^*v_s > 0$.  The proof follows in an identical 
fashion for $v_t^TS_t^*v_t > 0$.  This completes the proof.
\end{proof}


\section{Numerical Experiments}\label{sec:numericalexpts}

We demonstrate the results of the previous section through numerical experiments on a
simple SDP.  Consider the SDP with data
\beq
A_0 = \begin{bmatrix} 
1 & 1 & 1 & 0 \\
1 & 1 & 0 & 1 \\
1 & 0 & 1 & 1 \\
0 & 1 & 1 & 1 \end{bmatrix}, A_p = e_pe_p^T, b_p = 1 \;\forall\; p = 1,\ldots,4.
\label{example1}\eeq
This form of the SDP has the same structure as the SDP relaxation for MAXCUT 
investigated by Goemans and Williamson~\cite{GoeWil94}.  The eigenvalues and eigenvectors 
of $A_0$ are,
\[
\Lambda_0 = \begin{bmatrix} 
-1 & 0 & 0 & 0 \\
0 & 1 & 0 & 0 \\
0 & 0 & 1 & 0 \\
0 & 0 & 0 & 3 \end{bmatrix},
Q_0 = \begin{bmatrix} 
-\half & 0 & \frac{1}{\sqrt{2}} & -\half \\
\half & -\frac{1}{\sqrt{2}} & 0 & -\half \\
\half & \frac{1}{\sqrt{2}} & 0 & -\half \\
-\half & 0 & -\frac{1}{\sqrt{2}} & -\half 
\end{bmatrix}.
\]
Since $A_0$ has the smallest eigenvalue to be $-1$,  the optimal solution to the SDP 
defined by~\eqref{example1} is $X^* = 4 q_1q_1^T$ where $q_1$ is the first column of 
$Q_0$ (the eigenvector of $A_0$ corresponding to eigenvalue of $-1$).  The factor $4$ 
ensures that the equality constraints are satisfied.  

\subsection{Primal Non-degeneracy of $X^*$}
We show in the following that $X^*$ 
is primal non-degenerate by verifying satisfaction of~\eqref{nondegen-licq}.  
From the definition of ${\mathbb T}_X^{\perp}$ in~\eqref{orthtangentspace} we have,
\[
{\mathbb T}^{\perp}_X = \left\{ \left. Q \begin{bmatrix} 0 & 0 \\  0 & W \end{bmatrix}Q^T  \right| 
W \in \Sy^3 \right\}
\]
Suppose, there exists $(\alpha_1,\ldots,\alpha_4) \neq 0$ such that~\eqref{nondegen-licq} holds,
\beq
		\sum\limits_{p=1}^4 \alpha_p A_p \in {\mathbb T}^{\perp}_X \\
\implies	\sum\limits_{p=1}^4 \alpha_p Q^TA_pQ =  \begin{bmatrix} 0 & 0 \\  0 & W \end{bmatrix} \label{contra12c}
\eeq
for some $W \in \Sy^3$.  In order for~\eqref{contra12c} to hold the first column of the matrix on the left hand side 
of~\eqref{contra12c} must be $0$ for some $\{\alpha_p\} \neq 0$.  We show that such $\alpha_p$ does not 
exist.  The condition that the first column of $\sum_{p=1}^4 \alpha_p Q^TA_pQ$ is $0$ can be written as,
\[\begin{aligned}
0 &= \sum\limits_{p=1}^4 \alpha_p Q^TA_pQ e_1 = \sum\limits_{p=1}^4 \alpha_p (Q^Te_p)(e_p^Tq_1) \\
&= \sum\limits_{p=1}^4 (Q^Te_p)(\subblk{q}{1}{p}\alpha_p) = 
Q^T\begin{bmatrix} \alpha_1 \subblk{q}{1}{1} \\ \vdots \\ \alpha_4 \subblk{q}{1}{4} \end{bmatrix}.
\end{aligned}\]  
Since $Q^T$ is a non-singular matrix the above can only occur if $\subblk{q}{1}{p}\alpha_p = 0$ for all 
$p$.  Since $\subblk{q}{1}{p} \neq 0$ this implies that $\alpha_p = 0$ for all $p = 1,\ldots,4$.   
Thus,~\eqref{contra12c} does not hold for $\alpha_p \neq 0$ which proves the non-degeneracy of $X^*$.

\begin{figure}[h]
\centering
\subfigure[$\graph(\nodeset,\edgeset)$]{
\includegraphics[scale=0.4]{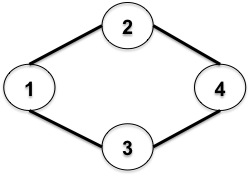}\label{fig:a}}
\subfigure[$\graph(\nodeset,\edgesetF)$]{
\includegraphics[scale=0.4]{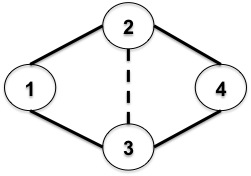}\label{fig:b}}
\subfigure[$\clique_1 = \{2,3,1\}, \clique_2 = \{2,3,4\}$]{
\includegraphics[scale=0.4]{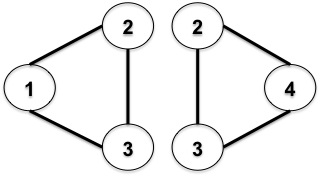}\label{fig:c}}
\caption{\subref{fig:a} Graph of the original SDP. \subref{fig:b} Graph of the chordal completion. 
\subref{fig:c} Maximal clique decomposition of chordal completion. }
\end{figure}

\subsection{Conversion SDP}

For the data in~\eqref{example1}, the graph of the aggregate sparsity pattern is 
depicted in Figure~\ref{fig:a}.  The $\graph(\nodeset,\edgeset)$ is a 4-cycle and not chordal.  
Figure~\ref{fig:b} shows a chordal extension where an edge $(2,3)$ has been introduced. 
The maximal clique decomposition for the chordal graph $\graph(\nodeset,\edgesetF)$ is shown 
in Figure~\ref{fig:c}.  Note that we have ordered the vertices such that~\eqref{stpermute} 
for ease of presentation. The conversion SDP is given by the data,
\[\bal
& \clique_1 = \{2,3,1\}, \clique_2 = \{2,3,4\} \\
& A_{1,0} = \begin{bmatrix}
1 & 0 & 1 \\
0 & 1 & 1 \\
1 & 1 & 1 \end{bmatrix}, A_{2,0} = \begin{bmatrix}
0 & 0 & 1 \\
0 & 0 & 1 \\
1 & 1 & 1 \end{bmatrix} \\
& A_{1,1} = e_3e_3^T, A_{2,1} = 0; A_{1,2} = e_1e_1^T, A_{2,2} = 0 \\
& A_{1,3} = e_2e_2^T, A_{2,3} = 0; A_{1,4} = 0, A_{2,4} = e_3e_3^T \\
& E_{1,22} = E_{2,22} = e_1e_1^T, E_{1,23} = E_{2,23} = \half(e_1e_2^T   + e_2e_1^T) \\
& E_{1,33} = E_{2,33} = e_2e_2^T.
\eal\] 
The solution to the conversion SDP is,
\[
X^*_1 = X^*_2 = \begin{bmatrix} 1 & 1 & -1 \\
1 & 1 & -1 \\
-1 & -1 & 1 \end{bmatrix} 
\]
Clearly, $\rank(X^*_1)$ = $\rank(X^*_2)$ = 1 $< |\clique_{12}|$.  Hence, 
Assumption~\ref{ass:smallrankofX} holds.  The eigenvectors and eigenvalues of $X^*_1$ are,
\[
\Lambda_1 = \begin{bmatrix} 3 & 0 & 0 \\
0 & 0 & 0 \\
0 & 0 & 0 \end{bmatrix}, Q_1 = \begin{bmatrix} 
\frac{1}{\sqrt{3}} & -\frac{1}{\sqrt{6}} & \frac{1}{\sqrt{2}}  \\
\frac{1}{\sqrt{3}} & -\frac{1}{\sqrt{6}} & -\frac{1}{\sqrt{2}}  \\
-\frac{1}{\sqrt{3}} & -\frac{2}{\sqrt{6}} & 0  \end{bmatrix}.
\]

\subsection{Primal Degeneracy}

As shown in Lemma~\ref{lem:eigenvectorXst} we have that $u = [\frac{1}{\sqrt{2}} \; 
\frac{1}{\sqrt{2}}]^T$ is a 0-eigenvector of the submatrix 
which corresponds to the intersection of the cliques, $\clique_{12}$.  As shown in 
Lemma~\ref{lem:eigenvectorXst}, $v_1 = v_2 = [u^T \; 0]^T$ are 
 0-eigenvectors of $X^*_1, X^*_2$ respectively.

From the definition of ${\mathbb T}^{\perp}_{s,X^*}$ it is easy to see that, 
\[\bal
& v_1v_1^T  = \begin{bmatrix} \half & -\half & 0 \\
-\half & \half & 0 \\
0 & 0 & 0 \end{bmatrix} = Q_1 \begin{bmatrix} 0 & 0 & 0 \\
0 & 0 & 0 \\
0 & 0 & 1 \end{bmatrix}Q_1^T \in {\mathbb T}^{\perp}_{1,X^*} \\
& v_1v_1^T = \half E_{1,22} - E_{1,23} + \half E_{1,33}. 
\eal\]
Similarly, it can be shown that 
\[\bal
& -v_2v_2^T  = Q_1 \begin{bmatrix} 0 & 0 & 0 \\
0 & 0 & 0 \\
0 & 0 & -1 \end{bmatrix}Q_1^T \in {\mathbb T}^{\perp}_{1,X^*} \\
& -v_2v_2^T = \half (-E_{2,22}) - (-E_{2,23}) + \half (-E_{2,33}). 
\eal\]
Thus, there exists an element in ${\mathbb T}^{\perp}_{1,X^*}$ and ${\mathbb T}^{\perp}_{2,X^*}$ 
that is in the span of the constraints that equate the 
elements in $\clique_{12}$.  Hence, the conversion SDP is primal degenerate.

\subsection{Non-unique Multipliers}

For the original SDP, the optimal multipliers are,
\[
\zeta^* = \begin{bmatrix} -1 \\ -1 \\ -1 \\ -1 \end{bmatrix}, 
S^* = \begin{bmatrix} 2 & 1 & 1 & 0 \\ 1 & 2 &  0 & 1 \\ 1 & 0 & 2 & 1 \\ 0 & 1 & 1 & 2 
\end{bmatrix}.
\]
For the conversion SDP, multipliers satisfying~\eqref{firstoptconvsdp} are
\[\bal
&\zeta^*_1 = \begin{bmatrix} -1 \\ -1 \\ -1 \\ 0 \end{bmatrix}, \zeta^*_2 = \begin{bmatrix} 
0 \\ 0 \\ 0 \\ -1 \end{bmatrix}, \xi^*_{12,22} = 1, \xi^*_{ 12,23} = 0, \xi^*_{12,33} = 1 
\eal\]
\[\bal
&S^*_1 = \begin{bmatrix} 1 & 0 & 1 \\ 0 & 1 & 1 \\ 1 & 1 & 2 \end{bmatrix}, 
S^*_2 = \begin{bmatrix} 1 & 0 & 1 \\ 0 & 1 & 1 \\ 1 & 1 & 2 \end{bmatrix}.
\eal\]
The eigenvectors of $S_1^*, S_2^*$ are $Q_1$ while the eigenvalues are,
\[
\Lambda_{S1} = \Lambda_{S2} = \begin{bmatrix} 0 & 0 & 0 \\ 0 & 3 & 0 \\ 0 & 0 & 1 \end{bmatrix}.
\]
Thus it is easy to see that $X^*_1S^*_1 = 0$ and they satisfy strict complementarity.  The  
same is also true of $X^*_2$ and $S^*_2$.  
The eigenvalue of $v_1$ is 1 and satisfies the conditions in 
Theorem~\ref{thm:nonuniquemultiplier} and hence, for all $0 \leq\delta \leq 1$,
\[\bal
& S^*_1-\delta v_1v_1^T \succeq 0, X^*_1( S^*_1-\delta v_1v_1^T) = 0  \\
& (\xi^*_{12,22}-\half\delta)E_{12,22} + (\xi^*_{12,23}+\delta)E_{12,23}  \\
& + (\xi^*_{12,33}-\half\delta)E_{12,33} - (S^*_1-\delta v_1v_1^T) \\
=\; & \xi^*_{12,22}E_{12,22} + \xi^*_{12,23} E_{12,23}  + \xi^*_{12,33} E_{12,33} 
- S^*_1
\eal\]
Further, it can also be shown that,
\[\bal
& (\xi^*_{12,22}-\half\delta)(-E_{12,22}) + (\xi^*_{12,23}+\delta)(-E_{12,23})  \\
& + (\xi^*_{12,33}-\half\delta)(-E_{12,33}) - (S^*_2+\delta v_2v_2^T) \\
=\; & -\xi^*_{12,22}E_{12,22} - \xi^*_{12,23} E_{12,23}  - \xi^*_{12,33} E_{12,33} 
- S^*_2 \\
& S^*_2+\delta v_2v_2^T \succeq 0, X_2^*(S^*_2+\delta v_2v_2^T) = 0.  
\eal\]
Thus, we have that the multipliers
\[\bal
& \zeta^*_1,\; \zeta^*_2,\; \xi^*_{12,22}-\half\delta,\; \xi^*_{12,23}+\delta,\; \xi^*_{12,33}-\half\delta \\
& S^*_1+\delta v_1v_1^T,\; S^*_2+\delta v_2v_2^T
\eal\]
also satisfy the first order optimality conditions for conversion SDP.  This shows that there 
are an infinite set of multipliers for the conversion SDP.

\subsection{Ill-conditioning in IPM}

Since the multipliers are not unique, the matrix used in the step 
computation of the IPM for SDP must be singular in the limit.  
Figure~\ref{fig:illcondn} plots the condition number of the 
schur-complement matrix in SDPT3~\cite{sdpt3} against the optimality gap.  
SDPT3 takes $7$ iterations to solve either formulation.  But the plot clearly shows that 
the condition number of the schur-complement matrix is higher for the conversion SDP.  
This is attributable to the non-uniqueness of the dual multipliers.  The plot shows that conditioning 
for the conversion SDP grows as $O(1/\mu^2)$ as opposed to $O(1/\mu)$ for non-degenerate SDPs.  
This observation is consistent with the analysis of Toh~\cite{Toh04}. 
\begin{figure}
\centering
\includegraphics[scale=0.35]{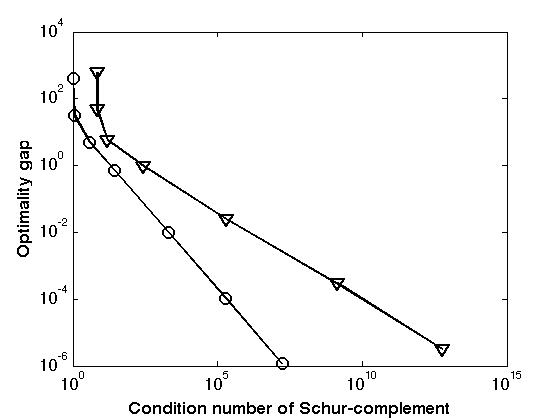}
\caption{Plot of the condition number of the  schur-complement matrix in the IPM against the 
optimality gap.  $\circ$ - original SDP formulation, $\triangle$ - conversion SDP.}\label{fig:illcondn}
\end{figure}

\section{Conclusions \& Future Work}\label{sec:conclusions}

We analyzed the conversion approach for SDP proposed by Fukuda et al~\cite{FukKojMur00}.  
The analysis showed that for SDPs with a low rank solution, the conversion SDP was  
primal degenerate.  We also provided conditions under which the multipliers for the 
conversion SDP were non-unique.  The theory was exemplified using a simple $4 \times 4$ 
SDP.  In the example, the ill-conditioning in the schur-complement matrix was 
greater for the conversion SDP.  Nevertheless, this did not affect the number of iterations to 
reach the said tolerance.  We believe the effect of the ill-conditioning is likely to be 
more dramatic for larger problems and affect convergence of IPM.  This will be investigated 
in a future study.


\end{document}